\documentclass[a4paper,10pt]{article}

\usepackage{geometry}
\usepackage{amsmath}
\usepackage{amssymb}
\usepackage{amsthm}
\usepackage{mathrsfs}
\usepackage{color}
\usepackage{amscd}
\usepackage{graphicx}
\usepackage{currfile}
\usepackage{hyperref}

\newtheorem{theorem}{Theorem}
\newtheorem*{theorem*}{Theorem}

\newtheorem{proposition}[theorem]{Proposition}
\newtheorem{corollary}[theorem]{Corollary}
\newtheorem{example}[theorem]{Example}

\theoremstyle{definition}
\newtheorem{remark}[theorem]{Remark}

\usepackage{amsfonts}

\newcommand{\hh}{{\mathbb{H}}}

\newcommand{\cc}{{\mathbb{C}}}
\newcommand{\rr}{{\mathbb{R}}}

\newcommand{\nn}{{\mathbb{N}}}

\newcommand{\bb}{{\mathbb{B}}}

\newcommand{\s}{{\mathbb{S}}}

\newcommand{\punto}{\cdot}

\newcommand\re{\operatorname{Re}}
\newcommand\im{\operatorname{Im}}

\newcommand{\Zt}{\widetilde{\mathcal Z}}
\newcommand{\dd[3]}{\frac{\partial #2}{\partial #3}}
\newcommand{\dcf}{\overline{\partial}_{\scriptscriptstyle CRF}}

\title{\bf A four dimensional Bernstein Theorem}

\author{Alessandro Perotti\\
\small Department of Mathematics, University of Trento\\ 
\small Via Sommarive 14, I-38123 Povo Trento, Italy\\
\small alessandro.perotti@unitn.it}

\date{  }

\begin{document}

\maketitle


\begin{abstract}
We prove a four dimensional version of the Bernstein Theorem, with complex polynomials being replaced by quaternionic polynomials. We deduce from the theorem a quaternionic Bernstein's inequality and give a formulation of this last result in terms of four-dimensional zonal harmonics and Gegenbauer polynomials. 

\end{abstract}


\section{Introduction}

In 1930, S.\ Bernstein \cite{Bernstein} proved the following result:
\begin{theorem*}[{\bfseries A}]
Let $p(z)$ and $q(z)$ be two complex polynomials with degree of $p(z)$ not exceeding that of $q(z)$. If $q(z)$ has all its zeros in $\{|z|\le1\}$ and $|p(z)|\le |q(z)|$ for $|z|=1$, then $|p'(z)|\le |q'(z)|$ for $|z| = 1$.
\end{theorem*}

From this result, the famous Bernstein's inequality (first established in this form by M.\ Riesz in 1914) can be deduced. Taking $q(z)=Mz^n$, one obtains the following
\begin{theorem*}[{\bfseries B}]
If $p(z)$ is a complex polynomial of degree $d$ and $\max_{|z|=1}|p(z)|=M$, then $|p'(z)|\le dM$ for $|z| = 1$.
\end{theorem*}

This note deals with a four dimensional version of such classic results, with complex polynomials being replaced by quaternionic polynomials. The extension of Bernstein's inequality to the quaternionic setting has already appeared in \cite{GalSabadini}. The proof given there is based on a quaternionic version of the Gauss-Lucas Theorem. Unfortunately, this last result is valid only for a small class of quaternionic polynomials, as it has been recently showed in \cite{GaussLucas}, where another version of the Gauss-Lucas Theorem, valid for every polynomial, has been proved. Recently, a different proof of the quaternionic Bernstein's inequality has been given in \cite{Xu}, using the Fej\'er kernel and avoiding the use of the Gauss-Lucas Theorem.


We refer the reader to \cite{GeStoSt2013} and \cite{DivisionAlgebras} for definitions and properties concerning the algebra $\hh$ of quaternions and many aspects of the theory of quaternionic \emph{slice regular} functions, a class of functions which includes polynomials and convergent power series.
The ring $\hh[X]$ of quaternionic polynomials is defined by fixing the position of the coefficients with respect to the indeterminate $X$ (e.g.\ on the right) and by imposing commutativity of $X$ with the coefficients when two polynomials are multiplied together (see e.g.\ \cite[\S 16]{Lam}). Given two polynomials $P,Q\in\hh[X]$, let $P\punto Q$ denote the product obtained in this way. 
A direct computation (see \cite[\S 16.3]{Lam}) shows that if $P(x)\ne0$, then
\begin{equation}\label{product}
(P\punto Q)(x)=P(x)Q(P(x)^{-1}xP(x)),
\end{equation}
while $(P\punto Q)(x)=0$ if $P(x)=0$. In particular, if $P$ has real coefficients, then $(P\punto Q)(x)=P(x)Q(x)$.
In this setting, a {(left) root or zero} of a polynomial $P(X)=\sum_{h=0}^dX^h a_h$ is an element $x\in\hh$ such that $P(x)=\textstyle\sum_{h=0}^dx^h a_h=0$. 

A subset $A$ of $\hh$ is called \emph{circular} if, for each $x\in A$, $A$ contains the whole set (a 2-sphere if $x\not\in\rr$, a point if $x\in\rr$)
\begin{equation}\label{sx}
\s_x=\{pxp^{-1}\in\hh\;|\;p\in\hh^*\},
\end{equation}
where $\hh^*:=\hh\setminus\{0\}$. In particular, for any imaginary unit $I\in\hh$, $\s_I=\s$ is the 2-sphere of all imaginary units in $\hh$. 
It it is well-known (see e.g.\ \cite[\S3.3]{GeStoSt2013}) that if $P\not\equiv0$, the zero set $V(P)$ consists of isolated points or isolated 2-spheres of the form \eqref{sx}.

We show that the quaternionic version of Theorem (A) holds true after imposing a necessary assumption on the second polynomial. We require that $Q\in\hh[X]$ has every coefficients belonging to a fixed subalgebra of $\hh$. This restricted version of the Bernstein Theorem is however sufficient to deduce the quaternionic Bernstein's inequality, i.e.\ the analog of Theorem (B). 
In Section \ref{sec:zonal}, we restate the inequality in terms of four-dimensional zonal harmonics and Gegenbauer polynomials. To obtain this form, we use results of \cite{Harmonicity} to obtain an Almansi type decomposition of a quaternionic polynomial.

\section{Bernstein Theorem and inequality}

Let $I\in\s$ and let $\cc_I\subset\hh$ be the real subalgebra generated by $I$, i.e.\ the complex plane generated by 1 and $I$. 
If $\cc_I$ contains every coefficient of $P\in\hh[X]$, then we say that $P$ is a \emph{$\cc_I$-polynomial}.
Every $\cc_I$-polynomial $P$ is \emph{one-slice-preserving}, i.e.\  $P(\cc_I)\subseteq\cc_I$. If this property holds for two imaginary units $I,J$, with $I\ne\pm J$, then it holds for every unit and $P$ is called \emph{slice-preserving}. This happens exactly when all the coefficients of $P$ are real.

Let $P(X)=\sum_{k=0}^dX^ka_k\in\hh[X]$ of degree $d\geq1$. Let $P'(X)=\sum_{k=1}^dX^{k-1}ka_k$ be the derivative of $P$.  For every $I\in\s$, let $\pi_I:\hh\to\hh$ be the orthogonal projection onto $\cc_I$ and $\pi_I^\bot=id-\pi_I$. Let $P^I(X):=\sum_{k=1}^dX^ka_{k,I}$ be the $\cc_I$-polynomial with coefficients $a_{k,I}:=\pi_I(a_k)$. 

We denote by $\bb=\{x\in\hh\,|\,|x|<1\}$ the unit ball in $\hh$ and by $\s^3=\{x\in\hh\,|\,|x|=1\}$ the unit sphere.

\begin{theorem}\label{thm:Bernstein}
Let $P,Q\in\hh[X]$ be two quaternionic polynomials with degree of $P$ not exceeding that of $Q$. Assume that there exists $I\in\s$ such that $Q$ is a $\cc_I$-polynomial. If $V(Q)\subseteq\overline\bb$ and $|P(x)|\le|Q(x)|$ for $x\in\s^3$, then $|P'(x)|\le|Q'(x)|$ for  $x\in\s^3\cap\cc_I$.
\end{theorem}

\begin{proof}
Let $\lambda\in\hh$ with $|\lambda|>1$ and set $R:=Q-P\lambda^{-1}\in\hh[X]$. The polynomials $Q$ and $R^I=Q-(P\lambda^{-1})^I$ are $\cc_I$-polynomials and then they can be identified with elements of $\cc_I[X]$, with $\deg(R^I)\le\deg(Q)$. For every $x\in\cc_I$, it holds
\[|R^I(x)-Q(x)|=|(P\lambda^{-1})^I(x)|=|\pi_I((P\lambda^{-1})(x))|\le|(P\lambda^{-1})(x)|=\frac{|P(x)|}{|\lambda|}.
\]
If $x\in\s^3\cap\cc_I=\{x\in\cc_I\,|\,|x|=1\}$, then
\begin{equation}\label{eq:inequality}
|R^I(x)-Q(x)|\le\frac{|P(x)|}{|\lambda|}\le\frac{|Q(x)|}{|\lambda|}\le |Q(x)|.
\end{equation}
In view of Rouch\'e's Theorem for polynomials in $\cc_I[X]$, $R^I$ and $Q$ have the same zeros in the disc $\{x\in\cc_I\,|\,|x|<1\}$. Moreover, if $|x|=1$ and $Q(x)=0$, the inequality \eqref{eq:inequality} gives $R^I(x)=0$. Since $\deg(R^I)\le\deg(Q)$ and $V(Q)\subseteq\overline\bb$, we get that $V(R^I)\cap\cc_I\subseteq \overline\bb\cap\cc_I$. From the classic Gauss-Lucas Theorem, we get $V(R')\cap\cc_I\subseteq V((R^I)')\cap\cc_I\subseteq \overline\bb\cap\cc_I$. 

Now let $x\in\cc_I$ with $|x|>1$ be fixed and define $\lambda:=Q'(x)^{-1}P'(x)\in\hh$. Observe that $Q'(x)\ne0$ again from the classic Gauss-Lucas Theorem applied to the polynomial $Q$ considered as element of $\cc_I[X]$. 
If $|\lambda|>1$, the polynomial $R=Q-P\lambda^{-1}\in\hh[X]$ defined as above has zero derivative at $x$: $R'(x)=Q'(x)-P'(x)\lambda^{-1}=0$, contradicting what obtained before. Therefore it must be $|\lambda|\le1$, i.e.\ $|P'(x)|/|Q'(x)|\le1$ for all $x\in\cc_I$ with $|x|>1$. By continuity, $|P'(x)|\le|Q'(x)|$ for all $x\in\cc_I$ with $|x|=1$.
\end{proof}

We recall that a quaternionic polynomial, as any slice regular function, satisfies the maximum modulus principle \cite[Theorem 7.1]{GeStoSt2013}.  Let 
 \[\|P\|=\max_{|x|=1}|P(x)|=\max_{|x|\le1}|P(x)|
 \]
 denote the sup-norm of the polynomial $P\in\hh[X]$ on $\bb$.

\begin{corollary}[{\bfseries Bernstein's inequality}]\label{cor:inequality}
If $P\in\hh[X]$ is a quaternionic polynomial of degree $d$, then $\|P'\|\le d\|P\|$.
\end{corollary}
\begin{proof}
Let $M=\|P\|$ and apply the previous theorem to $P(X)$ and $Q(X)=MX^d$. Since $Q$ is slice-preserving, the thesis of Theorem \ref{thm:Bernstein} holds for every $I\in\s$.
\end{proof}

The inequality of Corollary \ref{cor:inequality} is best possible with equality holding if and only if $P$ is a multiple of the power $X^d$.

\begin{proposition}
If $P\in\hh[X]$ is a quaternionic polynomial of degree $d$, and $|P'(y)|= d\|P\|$ at a point $y\in\s^3$, then $P(X)=X^da$, with $a\in\hh$, $|a|=\|P\|$.
\end{proposition}
\begin{proof}
We can assume that $P(X)$ is not constant. Let $b=P'(y)^{-1}$ and set $Q(X):=P(X)b=\sum_{k=1}^dX^ka_k$. Then $Q'(y)=1$, $\|Q\|=1/d$ and $\|Q'\|\le1$.
Let $I\in\s$ such that $\cc_I\ni y$. Then 
\[\textstyle 1=Q'(y)=\sum_k ky^{k-1}a_k=\pi_I(Q'(y))=\sum_k ky^{k-1}\pi_I(a_k)=(Q^I)'(y).
\]
If $x\in\cc_I\cap\s^3$, it holds
\[\textstyle\big|(Q^I)'(x)\big|=\big|\sum_k kx^{k-1}\pi_I(a_k)\big|=\big|\pi_I\big(\sum_k kx^{k-1}a_k\big)\big|\le\big|\sum_k kx^{k-1}a_k\big|=|Q'(x)|\le1.
\]
This means that  the $\cc_I$-polynomial $Q^I$, considered as an element of $\cc_I[X]$, satisfies the equality in the classic Bernstein's inequality. The same inequality implies that 
\[1=\max_{x\in\cc_I\cap\s^3}|(Q^I_{|\cc_I})'(x)|\le d\max_{x\in\cc_I\cap\s^3}|Q^I_{|\cc_I}(x)|\le d\|Q\|=1,
\]
i.e.\ $\max_{x\in\cc_I\cap\s^3}|Q^I_{|\cc_I}(x)|=1/d$. Therefore the restriction of $Q^I$ to $\cc_I$ coincides with the function $x^dc$, with $c\in\cc_I$, $|c|=1/d$:
\[Q^I(x)=\sum_{k=1}^d x^k\pi_I(a_k)=x^dc \text{\quad for every $x\in\cc_I$}.
\]
This implies that $\pi_I(a_d)=c$, $\pi_I(a_k)=0$ for each $k=1,\ldots,d-1$ and $Q$ can be written as $Q(X)=X^dc+\widetilde Q(X)$, with the coefficients of $\widetilde Q$ belonging to $\cc_I^\bot=\pi_I^\bot(\hh)$. When $x\in\cc_I\cap\s^3$, $\widetilde Q(x)\in\cc_I^\bot$, and then 
\[\frac1{d^2}\ge|Q(x)|^2=|x^dc|^2+|\widetilde Q(x)|^2=\frac1{d^2}+|\widetilde Q(x)|^2.
\]
This inequality forces $\widetilde Q$ to be the zero polynomial and then $P(X)=Q(X)b^{-1}=X^d cb^{-1}$.
\end{proof}

We now show that in Theorem \ref{thm:Bernstein}, the assumption on $Q$ to be one-slice-preserving is necessary.

\begin{proposition}\label{counterexample}
Let 
\[
P(X)=(X-i)\cdot(X-j)\cdot(X-k),\quad Q(X)=2X\cdot(X-i)\cdot (X-j).
\]
Then $V(Q)=\{0,i\}\subseteq\overline\bb$ and $|P(x)|\le|Q(x)|$ for every $x\in\s^3$, but there exists $y\in\s^3$ such that $|P'(y)|>|Q'(y)|$.
\end{proposition}
\begin{proof}
By a direct computation we obtain:
\begin{align}
P(X)&=X^3-X^2(i+j+k)+X(i-j+k)+1,\quad Q(X)=2X^3-2X^2(i+j)+2Xk,\\
P'(X)&=3X^2-2X(i+j+k)+i-j+k,\quad Q'(X)=6X^2-4X(i+j)+2k.
\end{align}
Let $P_1(X)=X-k$, $Q_1(X)=2X$, $P_2(X)=(X-j)\cdot P_1(X)$, $Q_2(X)=(X-j)\cdot Q_1(X)$. Then $P(X)=(X-i)\cdot P_2(X)$ and $Q(X)=(X-i)\cdot Q_2(X)$.
For every $x\in\s^3\setminus\{j\}$, using formula \eqref{product} we get
\[|P_2(x)|=|x-j||(x-j)^{-1}x(x-j)-k|\le2|x-j|=|x-j||2x|=|Q_2(x)|.
\]
Since $P_2(j)=Q_2(j)=0$, the inequality holds also at $j$. From this we obtain, for each $x\in\s^3\setminus\{i\}$,
\[|P(x)|=|x-i||P_2((x-i)^{-1}x(x-i))|\le |x-i||Q_2((x-i)^{-1}x(x-i))|=|Q(x)|.
\]
Since $P$ and $Q$ vanish at $i$, $|P(x)|\le|Q(x)|$ for every $x\in\s^3$.

Let $y=\frac1{10}\left(1+9i+4j-\sqrt2k\right)\in\s^3$. An easy computation gives
\[|P'(y)|^2=\frac7{25}(5+\sqrt2)\simeq 1.80, \quad |Q'(y)|^2=\frac4{25}(10-3\sqrt2)\simeq 0.92.
\]
\end{proof}

\section{Bernstein inequality and zonal harmonics}
\label{sec:zonal}

Since the restriction of a complex variable power $z^m$ to the unit circumference is equal to $\cos(m\theta)+i\sin(m\theta)$, the classic Bernstein inequality for complex polynomials can be restated in terms of trigonometric polynomials. In this section we show that a similar interpretation is possible in four dimensions, by means of an Almansi type decomposition of quaternionic polynomials and its relation with zonal harmonics in $\rr^4$.

Quaternionic polynomials, as any slice regular function, are biharmonic with respect to the standard Laplacian of $\rr^4$  \cite[Theorem 6.3]{Harmonicity}. In view of Almansi's Theorem (see e.g.\ \cite[Proposition 1.3]{Aronszajn}), the four real components of such polynomials have a decomposition in terms of a pair of harmonic functions. The results of \cite{Harmonicity} can be applied to obtain a refined decomposition of the polynomial in terms of the quaternionic variable. 

Let $\mathcal Z_{k}(x,a)$ denote the four-dimen\-sion\-al \emph{(solid) zonal harmonic} of degree $k$ with pole $a\in\s^3$ (see e.g.~\cite[Ch.5]{HFT}). The symmetry properties of zonal harmonics imply that $\mathcal Z_{k-1}(x,a)=\mathcal Z_{k-1}(x\overline a,1)$ for every $a\in\hh$ and any $a\in\s^3$. Moreover it holds \cite[Corollary 6.7(d)]{Harmonicity} 
\begin{equation}\label{eq:powers}
x^k=\widetilde{\mathcal Z}_k(x)-\overline x\, \widetilde{\mathcal Z}_{k-1}(x)\text{\quad  for every $x\in\hh$ and $k\in\nn$},
\end{equation}
where $\Zt_k(x)$ is the real-valued zonal harmonic defined by $\Zt_k(x):=\frac1{k+1}{\mathcal Z}_k(x,1)$ for any $k\ge0$ and by $\widetilde{\mathcal Z}_{-1}:=0$. 

In the following we will consider polynomials in the four real variables $x_0,x_1,x_2,x_3$ of the form $A(x)=\sum_{k=0}^d\widetilde{\mathcal Z}_k(x) a_k$, with quaternionic coefficients $a_k\in\hh$. They will be called \emph{zonal harmonic polynomials with pole 1}. 
All these polynomials have an axial symmetry with respect to the real axis: for every orthogonal transformation $T$ of $\hh\simeq\rr^4$ fixing 1, it holds $A\circ T=A$.

\begin{proposition}[{\bfseries Almansi type decomposition}]\label{teo:Almansi}
Let $P\in\hh[X]$ be a quaternionic polynomial of degree $d$. There exist two zonal harmonic polynomials $A$, $B$ with pole $1$, of degrees $d$ and $d-1$ respectively, such that 
\begin{equation}
P(x)=A(x)-\overline x B(x)\text{\quad for every $x\in\hh$.}\label{eq:Almansi}
\end{equation}
 The restrictions of $A$ and $B$ to the unit sphere $\s^3$ are spherical harmonics depending only on $x_0=\re(x)$.
\end{proposition}

\begin{proof}
Let $P(X)=\sum_{k=0}^dX^kc_k$. Formula \eqref{eq:Almansi} follows immediately from \eqref{eq:powers} setting
\[
A(x)=\sum_{k=0}^d\widetilde{\mathcal Z}_k(x)c_k\text{\quad and\quad}B(x)=\sum_{k=0}^{d-1}\widetilde{\mathcal Z}_{k}(x)c_{k+1}.
\]
The restriction of $\widetilde{\mathcal Z}_k(x)$ to the unit sphere $\s^3$ is equal to the Gegenbauer (or  Chebyshev of the second kind) polynomial $C^{(1)}_{k}(x_0)$, where $x_0=\re(x)$  (see \cite[Corollary 6.7(e)]{Harmonicity}). This property implies immediately the last statement.
\end{proof}

\begin{remark}
The zonal harmonics $A$ and $B$ of the previous decomposition can be obtained from $P$ through differentiation. Since $P(x)-P(\overline x)=A(x)-\overline xB(x)-A(\overline x)+xB(\overline x)=2\im(x)B(x)$, the function $B$ is the \emph{spherical derivative} of $P$, defined (see \cite{GhPe_AIM}) on $\hh\setminus\rr$ as $P'_s(x)=(2\im(x))^{-1}(P(x)-P(\overline x))$. In \cite{Harmonicity} it was proved that the spherical derivative of a slice regular function, in particular of a quaternionic polynomial, is indeed the result of a differential operation. Given the Cauchy-Riemann-Fueter operator 
\[\dcf =\dd{}{x_0}+i\dd{}{x_1}+j\dd{}{x_2}+k\dd{}{x_3},\]
 it holds $\dcf P=-2{P'_s}$. Therefore 
\begin{equation}\label{eq:AB}
A(x)=P(x)-\frac12{\overline x}\,\dcf P(x),\quad B(x)=-\frac12\dcf P(x).
\end{equation}
Defining $A$ and $B$ by formulas \eqref{eq:AB} and using results from \cite{Harmonicity}, it can be easily seen that the Almansi type decomposition  $f(x)=A(x)-\overline xB(x)$ holds true for every slice regular function $f$, with $A$ and $B$ harmonic and axially symmetric w.r.t.\ the real axis. Observe that $B=f'_s$ is the spherical derivative of $f$ and $A=f^\circ_s+x_0f'_s$, where $f^\circ_s(x)=\frac12(f(x)+f(\overline x))$ is the \emph{spherical value} of $f$ (see \cite{GhPe_AIM}).
\end{remark}

Thanks to the previous decomposition, the quaternionic Bernstein inequality of Corollary \eqref{cor:inequality} can be restated in terms of Gegenbauer polynomials  $C^{(1)}_{k}(x_0)$. Let $d\in\nn$. For any $(d+1)$-uple $\alpha=(a_0,\ldots,a_d)\in\hh^{d+1}$, let $Q_\alpha:\s^3\to\hh$ be  defined by 
\[Q_\alpha(x):=\sum_{k=0}^d (C^{(1)}_{k}(x_0)-\overline x\,C^{(1)}_{k-1}(x_0))a_k\]
 for any $x=x_0+ix_1+jx_2+kx_3\in\s^3$ (where we set $C^{(1)}_{-1}:=0$). Being the restriction to $\s^3$ of the quaternionic polynomial $P(X)=\sum_{k=0}^dX^ka_k$, which has biharmonic real components on $\hh$, $Q_\alpha$ is a quaternionic valued \emph{spherical biharmonic} of degree $d$ (see e.g.\ \cite{GrzebulaMichalik}).

\begin{corollary}
Let $\alpha=(a_0,\ldots,a_d)$ and $\alpha'=(a_1,2a_2,\ldots,ka_k,\ldots,da_d,0)\in\hh^{d+1}$.
Then it holds:
\[
\text{if \quad}|Q_\alpha(x)|=\left|\sum_{k=0}^d \left(C^{(1)}_{k}(x_0)-\overline x\,C^{(1)}_{k-1}(x_0)\right)a_k\right|\le M\text{\quad for every $x\in\s^3$},
\]
\[\text{then\quad}
|Q_{\alpha'}(x)|=\left|\sum_{k=0}^{d-1} \left(C^{(1)}_{k}(x_0)-\overline x\,C^{(1)}_{k-1}(x_0)\right) (k+1)a_{k+1}\right|\le dM\text{\quad for every $x\in\s^3$}.
\]
\end{corollary}
\begin{proof}
Let $P(X)=\sum_{k=0}^dX^ka_k$. From formula \eqref{eq:powers} it follows that the restriction of $P'$ to the unit sphere is the spherical biharmonic $Q_{\alpha'}$. Corollary \ref{cor:inequality} permits to conclude.
\end{proof}

\begin{remark}\label{rem:max}
Let $P\in\hh[X]$ be a polynomial with Almansi type decomposition $P(x)=A(x)-\overline xB(x)$ and let $y=\alpha+J\beta\in\s^3$, $\alpha,\beta\in\rr,\beta>0$.
Let $v=A(y)\overline{B(y)}$. It follows from general properties of slice functions \cite[Lemma 5.3]{DivisionAlgebras} that if $v\in\rr$, then $|P|_{|\s_y}$ is constant, while if $v\not\in\rr$, then 
the maximum modulus of $P$  on the 2-sphere $\s_y\subset\s^3$ is attained at the point $\alpha+I\beta$, with $I=\im(v)/|\im(v)|$, while the minimum modulus is attained at $\alpha-I\beta$.
In principle, this reduces the problem of maximizing or minimizing the modulus of $P$ on the unit sphere (or ball) to a one-dimensional problem.
\end{remark}

\begin{example}
Consider the polynomial 
$P(X)=(X-i)\cdot(X-j)\cdot(X-k)$
of Proposition \ref{counterexample}. Since the first four zonal harmonics are 
\[\Zt_0(x)=1,\ \Zt_1(x)= 2 x_0,\ \Zt_2(x)=3 x_0^2 - x_1^2 - x_2^2 - x_3^2,\ \Zt_3(x)=4 x_0 (x_0^2 - x_1^2 - x_2^2 -  x_3^2),
\]
the Almansi type decomposition of $P$ is $P(x)=A(x)-\overline xB(x)$, with
\begin{align*}
A(x)&=(1 + 4 x_0^3 - 4 x_0 x_1^2 - 4 x_0 x_2^2 - 4 x_0 x_3^2)+(i+j+k)(2 x_0 - 3 x_0^2 + x_1^2 + x_2^2 + x_3^2),
\\
B(x)&=(3 x_0^2 - x_1^2 - x_2^2 - x_3^2)+i(1-2x_0)-j (1 + 2 x_0) +k (1 - 2 x_0)
\end{align*}
harmonic polynomials. Their restrictions to $\s^3$ are the spherical harmonics
\begin{align*}
A_{|\s^3}(x)&=(1 - 4 x_0 + 8 x_0^3)+i(1 + 2 x_0 - 4 x_0^2)+j(1 - 2 x_0 - 4 x_0^2)+k(1 + 2 x_0 - 4 x_0^2) ,
\\
B_{|\s^3}(x)&=(-1+ 4 x_0^2)+i(1 - 2 x_0)-j(1 + 2 x_0)+k(1 - 2 x_0).
\end{align*}
Following the observation made in Remark \ref{rem:max}, since $\im(A(y)\overline{B(y)})=4((\alpha-1)i+\alpha k)$, where $\alpha=\re(y)$, $y\in\s^3$,
one can find the 2-sphere $\s_y\subset\s^3$ where the maximum modulus of $P$ is attained.  A direct computation gives $\re(y)=(1-\sqrt{19})/6\sim-0.56$ and the corresponding maximum value $\|P\|\sim4.70$ attained at the point $\tilde y=(1-\sqrt{19})/6-i(5+\sqrt{19})/12+k(1-\sqrt{19})/12$ of $\s^3$.

\end{example}

\begin{remark}
Some of the results presented in this note can be generalized to the general setting of real alternative *-algebras, where polynomials can be defined and share many of the properties valid on the quaternions (see \cite{GhPe_AIM}). 
The polynomials of Proposition \ref{counterexample} can be defined every time the algebra contains an Hamiltonian triple $i,j,k$, i.e.\ when the algebra contains a subalgebra isomorphic to $\hh$ (see \cite[\S8.1]{Numbers}). This is true e.g.\ for the algebra of octonions and for the Clifford algebras with signature $(0,n)$, with $n\ge2$. In all such algebras we can repeat the previous proofs and get the analog of Theorem \ref{thm:Bernstein}, as well as of the Bernstein inequality (see also \cite{Xu} for this last result). 
\end{remark}






\end{document}